\documentclass{article}
\usepackage[utf8]{inputenc}
\usepackage{amsmath, amsthm, amssymb, fdsymbol}
\usepackage{graphicx}
\usepackage{caption}
\usepackage{cite}
\usepackage[noend]{algorithmic}
\theoremstyle{definition}
\newtheorem{definition}{Definition}
\theoremstyle{plain}
\newtheorem{theorem}{Theorem}
\newtheorem{lemma}{Lemma}

\newtheorem{question}{Question}
\newtheorem{observation}{Observation}
\newtheorem{example}{Example}

\newtheorem{conjecture}{Conjecture}

\newcommand{\cP}{\mathcal{P}}

\usepackage{tikz}

\DeclareMathOperator{\press}{Press}
\DeclareMathOperator{\poset}{Poset}
\DeclareMathOperator{\linext}{LE}

\begin{document}

\title{A New Characterization of $\mathcal{V}$-Posets}
\author{Joshua Cooper, Peter Gartland, and  Hays Whitlatch}

\maketitle

\begin{abstract} 
In 2016, Hasebe and Tsujie gave a recursive characterization of the set of induced $N$-free and bowtie-free posets; Misanantenaina and Wagner studied these orders further, naming them ``$\mathcal{V}$-posets''.  Here we offer a new characterization of $\mathcal{V}$-posets by introducing a property we refer to as \emph{autonomy}.  A poset $\cP$ is said to be \emph{autonomous} if there exists a directed acyclic graph $D$ (with adjacency matrix $U$) whose transitive closure is $\cP$, with the property that any total ordering of the vertices of $D$ so that Gaussian elimination of $U^TU$  proceeds without row swaps is a linear extension of $\cP$. Autonomous posets arise from the theory of pressing sequences in graphs, a problem with origins in phylogenetics.  The pressing sequences of a graph can be partitioned into families corresponding to posets; because of the interest in enumerating pressing sequences, we investigate when this partition has only one block, that is, when the pressing sequences are all linear extensions of a single autonomous poset.  We also provide an efficient algorithm for recognition of autonomy using structural information and the forbidden subposet characterization, and we discuss a few open questions that arise in connection with these posets.
\end{abstract}

\begin{section}{Introduction}

A simple pseudo-graph is a graph that admits loops but not multiple edges (sometimes known as a ``loopy graph''). 
Given a simple pseudo-graph $G$, denote by $V(G)$ the vertex set of $G$; $E(G)\subseteq V(G)\times V(G)$, symmetric as a relation, its edge set. Let $N(v)=N_G(v)=\{w\in V(G):vw\in E(G) \}$ the \emph{neighborhood} of $v$ in $V(G)$.  Observe that $v\in N(v)$ iff $v$ is a looped vertex.  For $S\subset V$, we denote by $G[S]$ the vertex-induced subgraph on $S$.  

\begin{definition}\label{pressedgraph}
Consider a simple pseudo-graph $G$ with a looped vertex $v \in V(G)$.  ``Pressing $v$'' is the operation of transforming $G$ into $G'$, a new simple pseudo-graph in which $G[N(v)]$ is complemented.  That is, 
$$
V(G')=V(G), \quad E(G')=E(G)\triangle\left(N(v)\times N(v)\right)
$$
We denote by $G_{(v)}$ the simple pseudo-graph resulting from pressing vertex $v$ in $V(G)$ and we abbreviate $G_{(v_1)(v_2)\cdots (v_k)}$ to $G_{(v_1, v_2, \ldots ,v_k)}$. For $k\geq 1$ we abbreviate $(1,2,\ldots,k)$ as $\boldsymbol{k}$ so that 
when $V(G)=[n]$ for some $n\geq k$ then we may simplify $G_{(1,2,\ldots,k)}$ to $G_{\boldsymbol{k}}$.  $G_{\boldsymbol{0}}$ and $G_{()}$ are interpreted to mean $G$.  To aid with inductive arguments, we let $G^{(v)}=G_{(v)}-v$: the result of pressing $v$ in $G$ (which leaves it isolated, loopless, and thenceforth unpressable) and then removing the pressed vertex.
\end{definition}

Given a simple pseudo-graph $G$, $(v_1, v_2, \ldots,v_j)$ is said to be a \emph{successful pressing sequence} for $G$ whenever the following conditions are met:
\begin{itemize}
\item $\{v_1, v_2, \ldots,v_k\}\subseteq V(G)$,
\item $v_{i}$ is looped in $G_{(v_1, v_2, \ldots,v_{i-1})}$ for all $1\leq i\leq k$, 
\item $G_{(v_1, v_2, \ldots,v_k)}=(V(G), \emptyset)$
\end{itemize}
In other words, looped vertices are pressed one at a time, with ``success'' meaning that the end result (when no looped vertices are left) is an empty graph.  This topic originated in computational phylogenetics, where Hannenhalli and Pevzner showed that certain simple pseudo-graphs correspond to pairs of genomes and that the reversal edit distance between these genomes is the minimum length of a successful pressing sequence of said graph \cite{hannenhalli1999transforming}.   In  phylogenetics, the simple pseudo-graph corresponds to a pair of homologous genomes and its successful pressing sequences corresponds to a  most plausible (i.e., parsimonious) evolutionary history between the genomes (see \cite{dobzhansky1938inversions,sturtevant1936inversions}).  In the present work we look at the set of simple pseudo-graphs whose pressing sequences correspond to the linear extensions of a single poset.  Since linear extensions can be efficiently sampled asymptotically uniformly, this shows that pressing sequences, and hence the evolutionary histories of the pairs of genomes giving rise to said pseudo-graphs, can be sampled near-uniformly.

\begin{definition}
An \emph{ordered simple pseudo-graph}, abbreviated OSP-graph, is a simple pseudo-graph with a total order on its vertices.  In this paper, we will assume that the vertices of an OSP-graph are subsets of the positive integers under the usual ordering ``$<$''.  An OSP-graph $G$ is said to be \emph{order-pressable} if there exists some initial segment of $V(G)$ that is a successful pressing sequence.
\end{definition}
\begin{figure}[ht]
  \includegraphics[scale=.35]{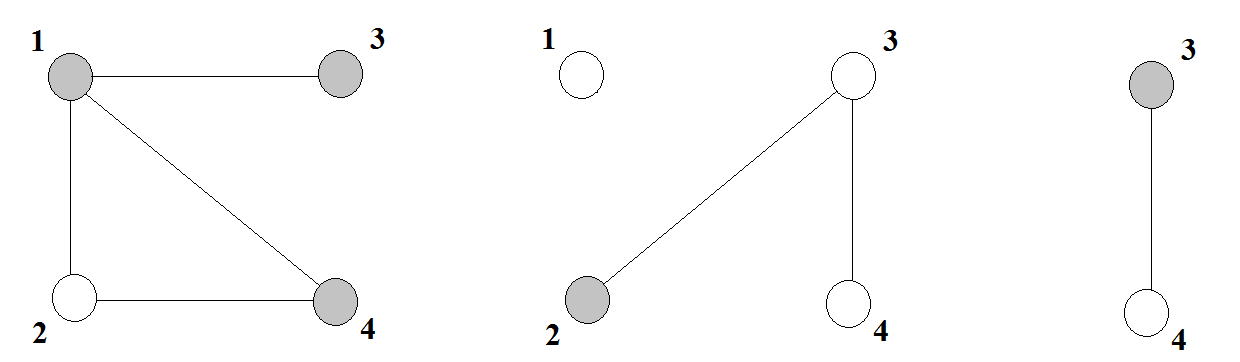}
  \caption{Left to right: an OSP-graph $G$; $G_{(1)}$, the result pressing $1$ in $G$;  and $G^{\boldsymbol{2}}$, the result of pressing and then removing vertices $1$ and $2$ in $G$. Loops are drawn a shaded vertices.  }
\end{figure}

\begin{definition}
It was shown in \cite{cooper2016successful} that pressing the vertices of a simple-pseudo-graph is essentially equivalent to performing Gaussian elimination with no row swaps on its adjacency matrix; therefore, the length of any successful pressing sequence of a simple pseudo-graph is the $\mathbb{F}_2$-rank of its adjacency matrix.  Thus, we define the \emph{rank of a simple pseudo-graph} to be the $\mathbb{F}_2$-rank of its adjacency matrix. The rank of a simple pseudo-graph on $n$ vertices can vary from $0$ (in the case that it is an edgeless simple pseudo-graph) to $n$ (such as is the case in Figure 1).  We say $G$ is \emph{full-rank} if its adjacency matrix is invertible over $\mathbb{F}_2$.  
\end{definition}

Call a matrix $M$ ``Cholesky"  if there exists an upper-triangular matrix $U$ so that $M = U^TU$. 
In \cite{cooper2016successful} a proof was given that Cholesky decompositions of full-rank, $\mathbb{F}_2$ matrices are unique; in  \cite{cooper2017uniquely} it was shown that for every OSP-graph and adjacency matrix $A$ there exists a particular Cholesky decomposition of $A$ that encodes the pressing instructions for $G$.
\begin{definition}
Let $G$ be OSP-graph with adjacency matrix $A$ (whose rows and columns are ordered by the identity permutation).  The \emph{instructional Cholesky root of $G$} (over $\mathbb{F}_2$) is the upper triangular matrix $U$ where for all $(i,j) \in [n] \times [n]$, 
$U[i,j]=1$ if and only if $ij \in E(G_{\mathbf{i-1}})$.  In \cite{cooper2017uniquely} it was shown that $U$ satisfies that $U^TU=A$, therefore is a Cholesky decomposition of $G$.
\end{definition}

The reason this matrix is called ``instructional'' is that it contains the instructions for how vertices affect one another during the corresponding pressing sequence: the $(i,j)$ entry is $1$ iff pressing $i$ flips the state of $j$.  Since the (instructional) Cholesky matrices are upper-triangular we may also regard $U$ as the adjacency matrix of a directed acyclic graph with vertex set $\{v\mid v$ is pressed at some point in the successful pressing sequence$\}$.  Furthermore, the transitive closure of this digraph can be considered as a poset.  Although it is possible to define these instructional posets for less-than-full-rank OSP-graphs, presently we are only concerned with the posets of full-rank OSP-graphs.

We refer to the set of looped vertices in a graph $G$ by $\mathcal{L}(G)$ and the set of successful pressing sequences for $G$ as $\Sigma(G)$.  

\begin{lemma}[{\bf\cite{cooper2016successful}, Theorem 9}]\label{pressifCholesky}
Let $G$ be a full-rank OSP-graph and $\sigma\in \Sigma(G)$.  Let $A$ be the adjacency matrix of $G$ with rows and columns ordered by $\sigma$. $\sigma\in \Sigma(G)$ if and only if $A$ has a Cholesky decomposition over $\mathbb{F}_2$.
\end{lemma}

\begin{definition}
Let $G$ be a full-rank OSP-graph and $\sigma \in \Sigma(G)$. Let $U$ be the instructional Cholesky root of A=adj$(G)$, with rows and columns ordered identically by $\sigma$,  and $D$ the digraph with  vertex set $V(G)$ and adjacency matrix $U$. The \emph{instructional poset of $G$ under $\sigma$} is $\poset(G,\sigma)=(V(G),\preceq)$ where $y\preceq x$ (equivalently $x\succeq y$) if there is an $x$ to $y$ path in $D$, i.e., $\poset(G,\sigma)$ is the transitive closure of $D$.
\end{definition}

We say \emph{$\cP$ is generated by $G$}, or equivalently \emph{$G$ is a generator of $\cP$},  if  $\poset(G,\sigma)=\cP$ for some $\sigma \in \Sigma(G)$.  If $\sigma$ is the natural order given by $G$ (typically the identity permutation) we simply write $\poset(G)$.  We denote the set of instructional posets of an OSP-graph $G$ by $\mathfrak{S}(G)$.  

\begin{center}
\begin{figure}[ht]
  \includegraphics[scale=.35]{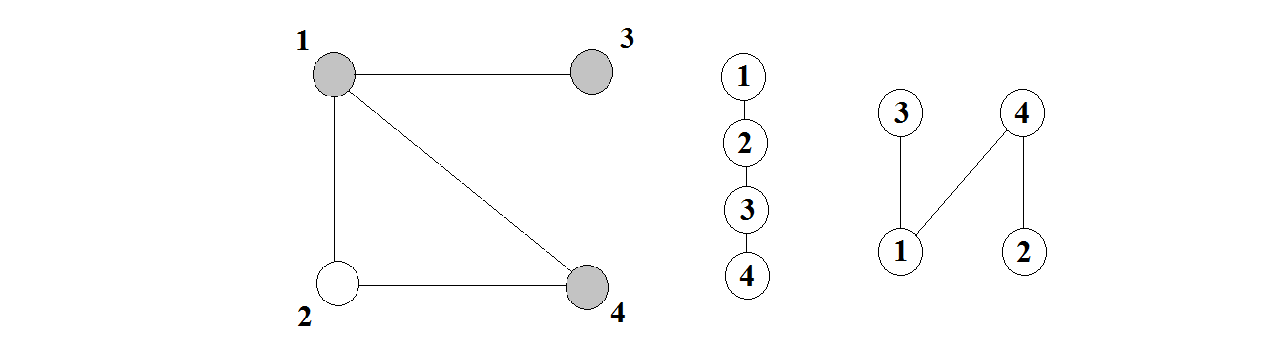}
  \caption{An order-pressable graph $G$ and the Hasse diagrams of the two posets it generates.}
\end{figure}
\end{center}
\begin{example}
Let $\cP$ be a poset on the element set $[4]=\{1,2,3,4\}$ with cover relations $1\succ 3$, $2\succ 3$, $3\succ 4$. Then any OSP-graph that generates $\cP$ must have an adjacency matrix $A=U^TU$ where $U$ is of the form $\begin{bmatrix}1&0&1&*\\0&1&1&*\\0&0&1&1\\0&0&0&1\\ \end{bmatrix}$. It follows that $\mathcal{P}$ has four generators, as shown below.

\begin{center}
\begin{figure}[ht]
  \includegraphics[scale=.46]{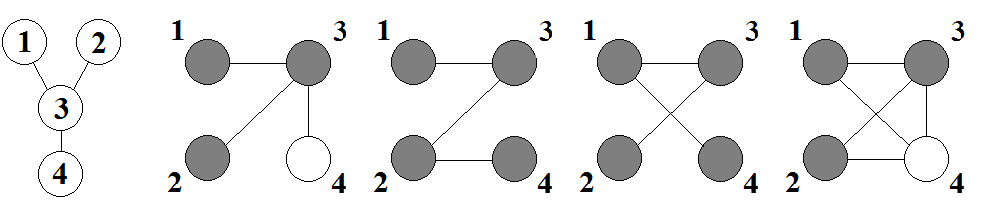}
  \caption{The Hasse diagram of $\cP$ and its four generators.}
\end{figure}
\end{center}

\end{example}

We finish this section with two more lemmas from \cite{cooper2016successful} which we will need below.

\begin{lemma}[\cite{cooper2016successful}, Proposition 1]\label{pressifloopedcomponents}
Let $G$ be an OSP-graph. $\Sigma(G)\neq \emptyset$ if and only if every component of $G$ containing two or more vertices contains a looped vertex.
\end{lemma}
\begin{lemma}[\cite{cooper2016successful}, Theorem 9]\label{pressifLPN}
Let $G$ be a full-rank OSP-graph and $\sigma\in \Sigma(G)$.  Let $A$ be the adjacency matrix of $G$ with rows and columns ordered by $\sigma$. $\sigma\in \Sigma(G)$ if and only if every leading principal minor (over $\mathbb{F}_2$) of $A$ is non-zero.
\end{lemma}

\end{section}
\begin{section}{Structure of Autonomous Posets}
We denote the set of linear extensions of a poset $\cP$ by $\linext(\cP)$.  
\begin{lemma}\label{linear extensions are pressing sequences}
If $G$ is a full-rank OSP-graph then $\linext(\cP(G,\sigma))\subseteq \Sigma(G)$ for all $\sigma\in \press(G)$. That is, $\Sigma(G)=\bigcup_{\cP \in \mathfrak{S}(G) }\linext(\cP) $.
\end{lemma}
\begin{proof}
Let $G=([n],E)$ be an OSP-graph of rank $n$ ordered by successful pressing sequence $\sigma$.  By relabeling $G$ we may assume $\sigma$ is the identity permutation.  Let $A$ be the adjacency matrix of $G$ (with rows and columns ordered by  $\sigma$) and $U$ be its instructional Cholesky root (identically ordered).  Let $D=([n],\overrightarrow{E})$ be the directed acyclic graph (aka ``DAG'') with adjacency matrix $U$.  
Let $\cP=\poset(G)=([n],\preceq_P)$ and observe that if $(a,b)\in \overrightarrow{E}$ then $a\succeq_p b$.

Fix a linear extension $\tau=(\tau_1, \tau_2, \ldots,\tau_n )$ of $\cP$.  By the previous observation,
 if $(\tau_i, \tau_j)\in   \overrightarrow{E}$ then $\tau_i\succeq_P\tau_j$ and hence $\tau_i$ must appear before $\tau_j$ in $\tau=(\tau_1, \tau_2, \ldots,\tau_n )$. 
 Thus, $(\tau_i, \tau_j)\in   \overrightarrow{E}$ implies $i\leq j \in \mathbb{N}$.   By contraposition, we have that $$i>j \textrm{ implies } (\tau_i, \tau_j)\notin   \overrightarrow{E}.$$
Let $P$ be the permutation matrix encoding $\tau$.  The previous assertion can be restated as  
$$\left[P^TUP\right]_{i,j}=0 \textrm{ for all }i>j.$$
Then $V=P^TUP$ is an upper-triangular matrix and $$V^TV=(P^TUP)^T(P^TUP)=P^TU^TUP=P^TAP.$$

Observe that $P^TAP$ is a full-rank symmetric matrix with a Cholesky decomposition given by $V$.    It follows from Lemma \ref{pressifCholesky} that $\tau$ is a successful pressing sequence for $G$.

\end{proof}

\begin{definition} 
We say an OSP-graph $G$ is an \emph{autonomous graph} if $\Sigma(G)=\linext(\poset(G))$.  We say  $\cP$ is an \emph{autonomous poset} if there exists an autonomous  graph $G$ that generates $\cP$.  That is, if there exists an OSP-graph $G$ such that $\poset(G,\sigma)=\cP$ for some  $\sigma\in \Sigma(G)$  and $\Sigma(G)=\linext(\cP)$.
\end{definition}
In our main theorem, we will show that the set of autonomous posets is precisely the set of induced $N$-free and induced bowtie-free posets (referred to in \cite{misanantenaina2018tutte} as ``$\mathcal{V}$-posets'').

\begin{definition}
For a graph $G$ and a vertex $x\notin V(G)$ we let $x\oplus G$ be the graph with vertex set $V(G)\cup \{x\}$, edge set $E(G)\triangle  \binom{\mathcal{L}(G)\cup \{x\}}{2}$, and $\mathcal{L}(x \oplus G)=\{x\}$.  Equivalently, $x\oplus G$ is the graph that results from adding a looped vertex $x$ to $V(G)$ and making it incident to each looped vertex in $G$ to get an intermediate graph $H$, then switching the state of each edge (including loops and non-loops) in $N_{H}(x)\setminus \{x\}$. We refer to this process as \emph{left-appending $x$ to $G$}, we justify this terminology in the following observation.   
\end{definition}
\begin{observation}\label{left append observation} \textrm{ }\\
Consider OSP-graphs $G$ and $H=x\oplus G$. Let $\tau=(\tau_1, \tau_2, \ldots, \tau_{n+1}) \in \Sigma(H)$.  Since  $\mathcal{L}(H)=\{x\}$ we have that $\tau_1=x$. Furthermore, pressing $x$ switches the state of every edge in $N_{H}(x)$ so $H^{(x)}=G$. Thus, the successful pressing sequences of $H$ are exactly those resulting from  left-appending $x$ to the successful pressing sequences of $G$.
If $G$ is order-pressable with instructional Cholesky root $U$, then $x\oplus G$ is order-pressable and has instructional Cholesky root $V$ that satisfies $$V[i,j]=\begin{cases} 
U[i-1,j-1] & \textrm{if $i,j\geq 2$}\\
1 & \textrm{if $i=1$ and $j\in \mathcal{L}(G)$}\\
0 & \textrm{otherwise.}\\
\end{cases} $$
\end{observation}

\begin{definition}
For a graph $G$ and a vertex $x\notin V(G)$ we let $ G \oplus x$ be the graph with vertex set $V(G)\cup \{x\}$, edge set $E(G)\cup \{ lx \mid l \in \mathcal{L}(G)\} $, and $$\mathcal{L}(G \oplus x)=\begin{cases}
\mathcal{L}(G) & \textrm{if $|V(G)|$ is odd}\\
\mathcal{L}(G)\cup \{x\} & \textrm{if $|V(G)|$ is even}\\
\end{cases}$$  
Equivalently, $G\oplus x$ is the graph that results from adding a vertex $x$ to $V(G)$,  making it incident to each looped vertex in $G$, and, if the resulting graph has an odd number of vertices, then we add a loop to $x$. We refer to this process as \emph{right-appending $x$ to $G$}.
\end{definition}

\begin{figure}[ht]
\includegraphics[scale=.35]{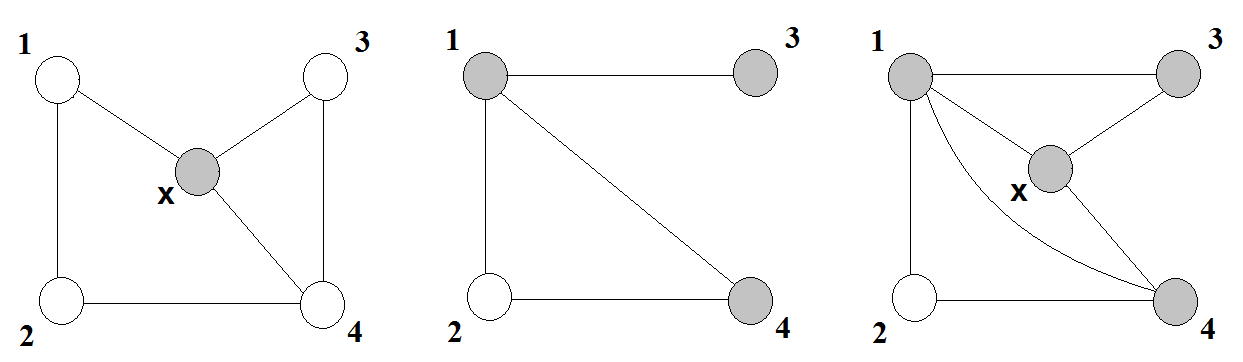}
  \caption{OSP-graphs $x\oplus G$, $G$, and $G\oplus x$, respectively.}
\end{figure}

Recall that the instructional Cholesky root of an OSP-graph is unique. In particular, if $H$ is a full-rank graph and $V^TV$ is a Cholesky factorization of $A=\textrm{adj}(H)$ then  $V$ must be the instructional Cholesky root of $H$; from this we get the following observation.
\begin{observation}\label{right append observation}\textrm{ }\\
If $G$ is order-pressable graph on $n$ vertices and has instructional Cholesky root $U$ then $G\oplus x$ is order-pressable and has instructional Cholesky root $V$ where $$V[i,j]=\begin{cases} 
U[i,j] & \textrm{if $i,j\leq n$}\\
1 & \textrm{if $j=n+1$}\\
0 & \textrm{otherwise.}\\

\end{cases}$$
\end{observation}

\begin{lemma}
If $G$ is autonomous then so is $x\oplus G$.
\end{lemma}
\begin{proof}
Let $H=x\oplus G$.  Since $\mathcal{L}(H)=\{x\}$ we have only one candidate vertex for an initial press.   Furthermore, by Observation \ref{left append observation}, $H^{(x)}=G$.  It follows that any pressing sequence must start with $x$ and then continue as a pressing sequence for $G$.  Therefore, the only instructional poset of $H$ is that of $G$ with a maximum element $x$ appended.  This demonstrates that $H$ is also autonomous.  
\end{proof}

\begin{lemma}
If $G$ is autonomous then so is $G\oplus x$.
\end{lemma}
\begin{proof}
If $|V(G)|=1$  and $G$ is order-pressable then $G$ is the graph on a single looped vertex and $G\oplus x$ is the graph with one looped vertex, one unlooped vertex and an edge between them; both of these graphs are uniquely pressable and therefore autonomous.  Assume now towards an inductive argument that $|V(G)|>1$ and that the inductive hypothesis holds for $|V(G)|-1$. Let $G=([n],E)$ and $H=G\oplus x$. 
By Observation \ref{right append observation}, every pressing sequence of $G$ can be extended to a pressing sequence for $H$ by appending $x$ to the end of the sequence. Therefore, we need only show that $|\Sigma(H)|=|\Sigma(G)|$ to conclude that $H$ generates only one poset, namely, $\poset(G)$ with the addition of a minimal element $x$.
Since $N_H(x)=\mathcal{L}(G)$, the result of pressing $x$ (should it be looped) in $H$ would be a loopless graph --  by Lemma \ref{pressifloopedcomponents} such a graph cannot be successfully pressed.   Thus, every successful pressing sequence for $H$ must begin with some element of $\mathcal{L}(H)\setminus\{x\}=\mathcal{L}(G)$.  Choose and fix $j \in \mathcal{L}(G)$ that is the initial vertex in a successful pressing sequence for $H$.  Assume, by way of contradiction, that $j$ is not maximal in $\poset(G)$.  It follows that no successful pressing sequence for $G$ begins with $j$, hence (by Lemma \ref{pressifloopedcomponents}) $G^{(j)}$ contains a loopless component on two or more vertices; call this component $C$.  

Consider now the result of pressing $j$ in $H$.  Since $$N_H(j)\cap V(C)= \mathcal{L}(G)\cap V(C)= N_H(x)\cap V(C),$$ we have that every edge from $x$ to $V(C)$ is deleted upon pressing $j$ and $V(C)$ is a set of unlooped vertices in $H^{(j)}$. Finally, observe that any vertex that is incident to $x$ in $H^{(j)}$ must be in a different component than $C$, as it was in $G$.   It follows that $H^{(j)}$ contains a non-trivial loopless component, contradicting that $j$ was the beginning of a successful pressing sequence.  Thus, the initial presses of $H$ are those of $G$. Observing that $H^{(j)}=G^{(j)}\oplus x$ the result follows from the inductive hypothesis.

\end{proof}

\begin{lemma}\label{remove min of poset}\label{autonomous are down closed} If $\cP$ is an autonomous poset and $k$ is a minimal element, then $\cP-k$ is also an autonomous poset. Furthermore, if $\mathfrak{S}(G)=\{\cP\}$ then $\mathfrak{S}(G-k)=\{\cP-k\}$.
\end{lemma}

\begin{proof}
Let $\cP$ is an autonomous poset on $n$ elements.  By relabeling, we may assume that the elements of $\cP$ are the integer set $[n]=\{1,2,\ldots, n\}$, so that $(1,2,\ldots, n)$ is a linear extension of $\cP$. By relabeling the minimal elements, we may assume the element we remove is $n$.

Let $G=([n],E)$ such that $G$ generates only $\cP$. Let $A$ be the adjacency matrix of $G$. By Lemma \ref{pressifLPN} and the fact that $\mathfrak{S}(G)=\{\cP\}$, for any permutation matrix $P$ we have that $P^TAP$ has all non-singular leading principal minors (i.e., is LPN) if and only if $P$ encodes a linear extension of $\cP$.
Let $A'$ denote the $(n-1)\times (n-1)$ leading principal submatrix of $A$.  Choose and fix an $(n-1)\times(n-1)$ permutation matrix $P'$.

Suppose  $P'^TA'P'$ is LPN.  Then $$\left[\begin{array}{c|c}
P' & 0\\ \hline 0 & 1\\
\end{array}\right]^T\left[\begin{array}{c|c}
A' & *\\ \hline * & a\\
\end{array}\right]\left[\begin{array}{c|c}
P' & 0\\ \hline 0 & 1\\
\end{array}\right]=\left[\begin{array}{c|c}
P'^TA'P' & *\\ \hline * & a\\
\end{array}\right]$$ is LPN  if and only if $\left[\begin{array}{c|c}
P'^TA'P' & *\\ \hline * & a\\
\end{array}\right]$ is invertible, which occurs if and only if $\left[\begin{array}{c|c}
A' & *\\ \hline * & a\\
\end{array}\right]$ is invertible.  Since $A$ is invertible, we may conclude that if $P'^TA'P'$ is LPN then $$\left[\begin{array}{c|c}
P' & 0\\ \hline 0 & 1\\
\end{array}\right]^T \cdot A \cdot \left[\begin{array}{c|c}
P' & 0\\ \hline 0 & 1\\
\end{array}\right]$$ is LPN.  It follows that every successful pressing sequence for a graph $G'$ with adjacency matrix $A'$ can be extended to a successful pressing sequence for $G$ by appending $n$ to the end of the sequence.  Furthermore, the instructional Cholesky root of $A'$ is the $(n-1)\times (n-1)$ leading principal submatrix of $A$; hence $G'$, the graph whose adjacency matrix is $A'$, generates $\mathcal{P}-n$.
\end{proof}

\begin{lemma}\label{remove max of poset}\label{autonomous are up closed} If $\cP$ is an autonomous poset and $k$ is a maximal element of $\cP$, then $\cP-k$ is also an autonomous poset.
\end{lemma}
\begin{proof}
Suppose $\cP$ is autonomous and $G$ is an OSP-graph such that $\mathfrak{S}(G)=\{\cP\}$.  
Let $U$ be the $n\times n$ intructional Cholesky root  of $G$.  Then the intructional Cholesky root of $G^{(1)}$ is the $(n-1)\times (n-1)$ trailing principal submatrix of $U$.  Thus, $G^{(1)}$ is a generator of $\cP-k$.  However, every successful pressing sequence of $G^{(1)}$ can be left-appended by $k$ to obtain a successful pressing sequences for $G$.  Hence, $\left|\Sigma\left(G^{(1)}\right)\right|=\left|\Sigma\left(G\right)\right|$, so that $\cP-k$ is the only poset generated by $G^{(1)}$.
\end{proof}

\begin{lemma}\label{max with leaf} Let $\cP$ be an autonomous poset on $n\geq 3$ elements. If $\cP$ has a maximum element $x$ and a minimal element $z$ such that $x$ covers $z$, then any graph $G$ that generates only $\cP$ must satisfy $|\mathcal{L}(G)|=1$.
\end{lemma}
\begin{proof}
By assumption that $x$ is maximum we have that $\cP$ is connected; therefore, if $y \in \cP\setminus\{x,z\}$, then $x\succ y$ and $y$ is incomparable to $z$. Suppose first that $n=3$, whence $\cP=(\{x,y,z\}, \preceq)$ with $x$ covering both $y$ and $z$.  If $G$ is an OSP-graph that generates $\cP$ then the adjacency matrix $A$ of $G$ must have an instructional Cholesky root $U$ encoding the cover relations of $\cP$.  Hence $$U=\begin{bmatrix}1 & 1 &1 \\0 &1&0\\0&0&1 \end{bmatrix} \quad \textrm{and so}\quad A=U^TU=\begin{bmatrix}1 & 1 &1 \\1 &0&1\\1&1&0 \end{bmatrix}.$$  

As the result holds for $n=3$, we proceed by induction on $n\geq 4$. Choose a minimal element $y \in \cP\setminus\{x,z\}$, let $\cP'=\cP-y$, and let $G'=G-y$.

By Lemma \ref{remove min of poset}, $\cP'$ is autonomous and  $\mathfrak{S}(G')=\{\cP'\}$.  Furthermore, $\cP'$ has a maximum element $x$ and a minimal element $z$ such that $x$ covers $z$, so we may apply the inductive hypothesis; $|\mathcal{L}(G')|=1$, in particular, $\mathcal{L}(G')=\{x\}$ (since it must be a pressable vertex). 
It follows that $\mathcal{L}(G)\subseteq \{x,y\}$. Assume, by way of contradiction, that $y\in \mathcal{L}(G)$.  If $xy\notin E(G)$ then pressing $y$ would create a looped vertex in every component of $G^{(y)}$, therefore there is a pressing sequence that begins with $y$, contradicting that $\cP$ is autonomous. Thus, we must conclude that $xy\in E(G)$.  Since $z$ is a minimal element covered by $x$, then $z$ is an isolated looped vertex in $G^{(x)}$ and hence $N_G(z)= N_G(x)$.  In particular, $yz\in E(G)$.  

Let $S=N_G(x)\setminus N_G(y)$ and $T=N_G(y) $.  Assume, towards a contradiction, that $S\neq \emptyset$.  Observe that $sx, sz\in E\left(G^{(y)}\right)$ for all $s\in S$ and hence there is a connected component in $G^{(y)}$ containing $x$ and $z$ (as well as the elements of $S$), and $z$ is looped in $G^{(y)}$.   Every other connected component in $G^{(y)}$ was created by deleting an edge between the vertices of $T$ and hence contains an element of $T$ which is now looped. It follows that $G^{(y)}$ can be successfully pressed, which is a contradiction.  Thus, we may proceed under the assumption that $S=\emptyset$.

If $v\in N_G(y)\setminus N_G(x)$ then $v$ is looped in $G^{(y)}$, $xy\in E\left(G^{(y)}\right)$, and every other connected component in $G^{(y)}$ was created by deleting the edge between two unlooped vertices and therefore would contain a looped vertex.  It follows that $N_G(y)= N_G(x)$, therefore $x$ and $y$ can be interchanged in any successful pressing sequence.  This contradicts that $G$ is autonomous, so  we must conclude that $y\notin \mathcal{L}(G)$, as desired.
\end{proof}

\begin{definition}
Let $\cP=(X,\preceq)$ be a poset.  We say $(a,b,c,d)$ is an \emph{occurrence of the pattern $N$} in $\cP$ if $\{a,b,c,d\}\subseteq X$ and $a\succ c$, $a\succ d$, and $b\succ d$.  We say $(a,b,c,d)$ is an \emph{induced occurrence} of the pattern $N$ in $\cP$ if $a\succ c$, $a\succ d$, $b\succ d$ and otherwise $a,b,c$ and $d$ are pairwise incomparable. 

We say $(a,b,c,d)$ is an \emph{occurrence of the pattern bowtie} in $\cP$ if $\{a,b,c,d\}\subseteq X$ and $a\succ c$, $a\succ d$, $b\succ c$, and $b\succ d$.  We say $(a,b,c,d)$ is an \emph{induced occurrence} of the pattern bowtie in $\cP$ if $a\succ c$, $a\succ d$, $b\succ d$ and otherwise $a,b,c$ and $d$ are pairwise incomparable. 

We say $\cP$ is \emph{induced $N$-free} if it contains no induced occurrences of the pattern $N$. Similarly, $\cP$ is \emph{induced bowtie-free} if it contains no induced occurrences of the pattern bowtie.
\end{definition}
It is worth noting that the literature varies on the definitions of ``$N$-free poset''.  In our terminology a poset may include an occurrence of the pattern $N$ yet be induced $N$- and bowtie-free.  Such an example is the poset $\cP=([4],\{1\succ 2\succ 3\succ 4\})$.

\begin{lemma}\label{N not an autonomous} Autonomous posets are induced $N$-free.
\end{lemma}
\begin{proof}
Let $\cP'$ be an autonomous poset and assume towards a contradiction $(a,b,y,z)$ is an induced occurrence of the pattern $N$ in $\cP'$. Let $\cP=(X,\preceq)$ be the result of iteratively removing maximal and minimal elements from $\cP'$ until $a,b$ are the only maximal elements and $y,z$ are the only minimal elements.  By Lemmas \ref{remove min of poset} and \ref{remove max of poset},  $\cP$ is an autonomous poset with an induced occurrence of the pattern $N$, namely $(a,b,y,z)$.  Observe that if there exists $(a',b',y', z')\neq (a,b,y,z)$ that induces the pattern $N$ in  $\cP$  then we may repeat the process of iteratively removing elements until only $a',b',y',z'$ are extremal elements;  thus, we proceed under the  assumption that  $\cP$ has exactly one induced occurrence of the pattern $N$. 

Choose $x\in \cP$ such that $x\succ y$ (hence $x\neq y$).  By assumption that only $a$ and $b$ are maximal in $\cP$ we have that $a\succeq x$ or $b\succeq x$.  Since $(a,b,y,z)$ is an induced occurrence of the pattern $N$ we have $b\not\succ y$ and hence $b\not\succeq x$, therefore $a\succeq x$.  Observe that if $x\not\succeq z$ then $(a,b,x,z)$ is an induced occurrence of the pattern $N$, contrary to assumption.  Thus, $x \succ z$ (since $x\neq z$) and it follows that $(x,b,y,z)$ is an induced occurrence of the pattern $N$ implying that $x=a$,  therefore $a$ covers $y$.

Now choose $w\in \cP$ such that $b\succ w$, observe that $w\neq a$.  Since $b\not\succ y$ we have $w\not\succeq y$, hence $w\succeq z$.  If $a\succ w\succ z$ then $(a,b,y,w)$ is an induced occurrence of the pattern $N$, contrary to assumption. Hence,  $a\succeq w$ if and only if $w=z$.  However, if $w\neq z$ then $(a,w,y,z)$ is an induced occurrence of the pattern $N$, again contrary to assumption.  Therefore, $w=z$ and it follows that $b$ covers $z$.

By assumption that $\cP$ is autonomous there exists a graph $G$ that generates only $\cP$.  Fix such a $G$.  Since $b\in \cP$ is maximal, there is a successful pressing sequence beginning with $b\in V(G)$; thus $b\in \mathcal{L}(G)$.   A sequence $\sigma'=(\sigma_1, \ldots, \sigma_k)$ is successful in $G^{(b)}$ exactly when $\sigma=(b,\sigma_1, \ldots, \sigma_k)$ is successful in $G$.  Since $G$ generates an autonomous poset then so does $G^{(b)}$ and hence $\cP-b$ is autonomous.  Further $\cP-b$ meets the description of Lemma \ref{max with leaf} so $\mathcal{L}\left(G^{(b)}\right)=\{a\}$,  therefore $\mathcal{L}\left(G\right)=\{a,b\}\cup N_G(b)$.  Now observe that if $v\in N_G(b)$, then pressing $b$ affects $v$ and hence $b\succeq v$.  It follows that $N_G(b)=\{b,z\}$,  therefore   $\mathcal{L}\left(G\right)=\{a,b,z\}$. We proceed to show that $z$ can be pressed in $G$, contradicting that  $\mathfrak{S}(G)=\{\cP\}$

Suppose first that $a\notin N_G(z)$.  Then $N_G(z)\setminus \{b\}\subseteq \mathcal{L}\left(G^{(z)}\right)$ and $bv\in E\left(G^{(z)}\right)$ for all $v\in N_G(z)\setminus \{b\}$.   It follows that any component created by pressing $z$ in $G$ has a looped vertex, and hence there is a successful pressing sequence starting with $z$ in $\Sigma(G)$, a contradiction.  Thus we must conclude that $\{a,b,z\}\subseteq N_G(z)$.  Observe that the only elements comparable to $y$ in $\cP$ are $a$ and $y$ itself.  Thus in any successful pressing sequence of $G$, $a$ must be pressed before $y$ and no other vertex affects (or is affected by) $y$.  Hence $y \notin \mathcal{L}(G)$ and $N_G(y)=N_G(a)\setminus \{y\}$.  Then $\{a,b,y,z\}\subseteq  N_G(z)$.  Since $ab,by\notin E(G)$ we have that $ab,by \in E\left(G^{(z)}\right)$ and hence $a, b$ and $y$ are path connected and $y \in \mathcal{L}\left(G^{(z)}\right)$.   Similarly, if $v\in N_G(z)\setminus \{a,b,y,z\}$ then $bv \in E\left(G^{(z)}\right)$.  It follows that every non-trivial component created by pressing $z$ in $G$ contains a looped vertex, therefore $z$ is the initial press of for some $\sigma \in \Sigma(G)$, a contradiction.

\end{proof}

Before proceeding, we state the main theorem of \cite{cooper2017uniquely}, which will be used below.

\begin{theorem}[{\bf\cite{cooper2017uniquely}, Theorem 1}]\label{uniquelymainthm}
Let $G=([n],E)$ be full rank with instructional Cholesky root $U$. Then $G$ is uniquely pressable (i.e., has exactly one pressing sequence) if and only if $U$ has columns $C_1$, $\ldots$, $C_n$ whose weights (number of nonzero entries) are $w_1$, $\ldots$, $w_n$ respectively, satisfying:
\begin{itemize}
\item For each $j$, if $C_j=(c_{1,j}, c_{2,j}, \ldots, c_{n,j})^T$ then $\begin{cases} c_{i,j}= 1, &  j-w_j < i\leq j\\
c_{i,j}= 0, &  \textrm{ otherwise}\\
\end{cases}$.
\item $1=w_1\leq w_2 \leq \cdots \leq w_n$.
\item $w_i>2$ implies $w_{i+2}>w_i$, for $i\in [n-2]$.
\item If $w_i$ is odd for $i>1$, then $w_j = j$ for all $j \geq i$.
\end{itemize}
\end{theorem}

For an integer $n$, let $\Lambda(n)$ denote the poset with element set $[n]$ such that $n-2$ covers $n$ and $i$ covers $i+1$ for all $i\in [n-2]$.  The Hasse diagram of $\Lambda(n)$ consist of two minimal elements ($n-1$ and $n$) below a chain of length $n-2$. Let $G_{\Lambda(n)}$ be the OSP-graph with vertex set $V(G)=[n]$, edge set $E(G)=\{ (i,i+1) \mid i\in [n-1]\}\cup\{ (1,1),(n-2,n) \}$.  

\begin{lemma}\label{unique autonomous} $\Lambda(n)$ is  an autonomous poset and $G_{\Lambda(n)}$ is the unique graph which generates only $\Lambda(n)$.
\end{lemma}
\begin{proof}
Observe that for $n=3$ we have only one instructional Cholesky that generates $\Lambda(n)$; $$U=\begin{bmatrix}
1 &1&1\\0&1&0\\0&0&1
\end{bmatrix} .$$ It follows that the only graph that generates $\Lambda(3)$ has adjacency matrix $$A=U^TU= \begin{bmatrix}
1 &1&1\\1&0&1\\1&1&0
\end{bmatrix} $$ which is the adjacency matrix of $G_{\Lambda(3)}$.

For $n=4$ we need only consider instructional Cholesky roots of the form: $$\begin{bmatrix}
1 &1 &*_1 &*_2\\0&1 &1&1\\0&0&1&0\\0&0&0&1
\end{bmatrix} $$  where $*_1, *_2\in \{0,1\}$.  A quick check reveals that setting $*_1= *_2=0$ yields a graph with two successful pressing sequences $(1,2,3,4)$ and $(1,2,4,3)$, and otherwise the resulting graph has $3$ or more successful pressing sequences; hence the claim holds for $n=4$.

We proceed by induction on $n\geq 5$.  Let $G$ be an OSP-graph that generates only $\Lambda(n)$.  Since $\Lambda(n)$ has maximum element $1$, we have that $1\in \mathcal{L}(G)$ and $G^{(1)}$ has instructional poset $\Lambda(n)-1$.  But $\Lambda(n)-1$ is
isomorphic to $\Lambda(n-1)$.  By the inductive hypothesis we have that $G^{(1)}$ is isomorphic to $G_{\Lambda(n-1)}$.

Let $U$ be the instructional Cholesky of $G$ under the identity permutation.  Let $A=U^TU$ and let $U'$ be the $(n-1) \times (n-1)$ leading principal submatrix of $U$,  $A'=U'^TU'$ and $G'=([n-1],E')$ the graph with adjacency matrix $A'$.  
Choose and fix $\sigma \in S_n$ such that $\sigma(n)=n$ and let $P_{\sigma}$ be the permutation matrix encoding $\sigma$.  Let $P_{\sigma'}$  be the $(n-1)\times (n-1)$ leading principal submatrix of $P_{\sigma'}$ and $\sigma'$ its corresponding permutation. 
Observe that since $G$ is full-rank then $A$ is invertible. Hence, $\sigma'\in \Sigma(G')$ if and only if $P_{\sigma'}^TA'P_{\sigma'} \textrm{ is in LPN form},$ which occurs if and only if $P_{\sigma}^TAP_{\sigma} \textrm{ is in LPN form}$, which in turn occurs if and only if 
$\sigma\in \Sigma(G).$

Since $\Sigma(G)=\{(1,2,\ldots,n-2, n-1, n), (1,2,\ldots,n-2, n, n-1)\}$ we have that the only successful pressing sequence of $G'$ is $\sigma'=(1,2,\ldots,n-2, n-1)$ and hence $G'$ is a uniquely pressable graph (has only one pressing sequence). By Theorem \ref{uniquelymainthm}, if $U'[1,i]=1$ then $U'[2,i]=U'[2,i+1]=1$ and hence 
for $2\leq i\leq n-2$ if $U[1,i]=1$ then $U[2,i]=U[2,i+1]=1$ .  However the intructional Cholesky root of $G^{(1)}$ is the $(n-1)\times (n-1)$ trailing principal minor of $U$ and $G^{(1)}$ is isomorphic to $G_{\Lambda(n-1)}$.  It follows that $U[2,i+1]=0$ for all $3\leq i\leq n-1$ thus $U[1,i]=0$ for all $3\leq i\leq n-1$, hence $[3,n-1]\cap N_G(1)=\emptyset$.  Observe that by relabeling $n$ to $n-1$ and vice-versa we can make the same argument and conclude that $n\notin N_G(1)$, therefore $U[1,n]=0$.  We conclude that $G=G_{\Lambda(n)}$.

\end{proof}

For an integer $n$ we let $X(n)$ denote the poset with element set $[n]$ so that $1$ covers $3$, $n-2$ covers $n$, and $i$ covers $i+1$ for all $i\in [2,n-2]$.  The Hasse diagram of $X(n)$ consist of a chain of length $n-4$ joining two minimal elements ($n-1$ and $n$) to two maximal elements ($1$ and $2$). 

\begin{lemma}\label{X poset not an autonomous} $X(n)$ is not an autonomous poset.
\end{lemma}
\begin{proof}
Assume, by way of contradiction, that $X(n)$ is  an autonomous poset and let $G$ be any graph that generates only $X(n)$.  Every successful pressing sequence of $G$ must begin with $1,2,3$ or $2,1,3$.  Thus, $\{1,2\}\subseteq \mathcal{L}(G)$. Since $3$ must be looped after pressing $1$ and $2$, and since the instructional Cholesky root instructs that both $1$ and $2$  switch the state of $3$ upon being pressed, then $3\in \mathcal{L}(G)$. 
Observe that $X(n)-1$ and $X(n)-2$ are isomorphic to $\Lambda(n-1)$ and hence $G^{(1)}$ and $G^{(2)}$ are isomorphic to $G_{\Lambda(n-1)}$ and hence each have exactly one looped vertex.  In particular,  $\mathcal{L}\left(G^{(i)}\right)=\{j\}$ for $\{i,j\}=\{1,2\}$.  Since $1$ and $2$ are both maximal in $X(n)$ then $(1,2)\notin E(G)$.  It follows that  $N_G(j)=N_{G^{(i)}}(j)$ for $ \{i,j\}=\{1,2\}$.  Therefore, by considering the structure of $G_{\Lambda(n-1)}$, we see $N_G(1)\setminus\{1\}=N_G(2)\setminus\{2\}=\{3\}$; furthermore, $\mathcal{L}(G)=\{1,2,3\}$.  Consider the result of pressing $3$ in $G$:  $(1,2),(1,4),(2,4)$ become edges, $4$ becomes looped, and every other vertex incident to $3$ in $G$ becomes incident to both $1$ and $2$ in $G^{(3)}$.  Thus, there is exactly one component in $G^{(3)}$ and it contains a looped vertex at $4$.  By Lemma \ref{pressifloopedcomponents} there is a successful pressing sequence in $G$ that begins with $3$, a contradiction.  We conclude that $X(n)$ is not an autonomous poset.

\end{proof}

\begin{lemma}\label{bowties are not autonomous} Autonomous posets are induced bowtie-free.
\end{lemma}
\begin{proof}
Let $\cP$ be an autonomous poset.  By Lemma \ref{N not an autonomous}, $\cP$ is induced $N$-free.  Assume, towards a contradiction, that $(a,b,y,z)$ is an induced occurrence of the pattern bowtie.  By iteratively removing maximal and minimal elements, and by application of Lemmas \ref{remove min of poset} and \ref{remove max of poset},  we may assume $a,b,y$, and $z$ are the only extremal elements of $\cP$, and that $\cP$ does not properly contain another occurrence of the pattern bowtie.

If the only elements of $\cP$ are $a,b,y,z$ then $$U=\begin{bmatrix}
1 &0 &1 &1\\0 &1 &1 &1\\0 &0 &1 &0\\0 &0 &0 &1\\
\end{bmatrix}$$ and hence $$A=\begin{bmatrix}
1 &0 &1 &1\\0 &1 &1 &1\\1 &1 &1 &0\\1 &1 &0 &1\\
\end{bmatrix}$$ which has a successful pressing sequence of $(4,3,2,1)$, contrary to assumption.

Choose and fix $x\in \cP$ such that $x\notin \{a,b,y,z\}$.  Since $x$ is not extremal in $\cP$ we may assume, without loss of generality, that $a\succ x\succ y$.  If $b\succ x \not\succ z$ then $(a,b,x,z)$ induces a bowtie, contrary to assumption.  Similarly, if $b\not\succ x \succ z$ then $(x,b,y,z)$ induces a bowtie. Observe that if $b\not\succ x \not\succ z$ then $(a,b,x,z)$ induces an $N$, contradicting Lemma \ref{N not an autonomous}.  Thus we must proceed under the assumption that $b\succ x \succ z$.  

Observe that the choice of $x$ was arbitrary so any $w\in \cP\setminus \{a,b,y,z\}$ must also satisfy $a\succ w \succ y$ and $b\succ w \succ z$.  If $x$ and $w$ are incomparable then $(a,b,x,w)$ and  $(x,w,y,z)$ induce a smaller bowtie, contrary to assumption.  Hence, any two elements in $\cP\setminus \{a,b,y,z\}$ must be comparable, therefore $\cP=X(m)$ for some $m\geq 5$.  This contradicts Lemma \ref{X poset not an autonomous}.

\end{proof}

\end{section}
\begin{section}{Main Result}
In \cite{hasebe2017order}
(and later in \cite{misanantenaina2018tutte}) the authors gave a simple description of posets  that are both induced $N$-free and induced bowtie-free which we include here as Definition \ref{V definition} and Theorem \ref{V theorem}.

\begin{definition}\label{V definition}
A poset is called a \emph{$\mathcal{V}$-poset} if it can be generated by beginning with the singleton poset and then iteratively applying any of the following  three operations:
\begin{itemize}
\item[(1)] a disjoint union,
\item[(2)] adding a new greatest element,
\item[(3)] adding a new least element.
\end{itemize}
\end{definition}

\begin{theorem}[\cite{hasebe2017order}, Theorem 4.3]\label{V theorem}
A poset is induced $N$-free and induced bowtie-free if and only if it is a $\mathcal{V}$-poset.
\end{theorem}

\begin{theorem}
$\cP$ is autonomous if and only if $\cP$ is induced $N$-free and induced bowtie-free.
\end{theorem}
\begin{proof} 
By Lemmas \ref{N not an autonomous} and \ref{bowties are not autonomous},  if $\cP$ is autonomous then $\cP$ is induced $N$-free and induced bowtie-free. By Theorem \ref{V theorem} it suffices to show that $\mathcal{V}$-posets are autonomous.

A poset on one element is autonomous as it corresponds to the uniquely pressable graph on a single looped vertex.  We proceed by induction.  Let $n\geq 2$ and assume that all  $\mathcal{V}$-posets on $n-1$ vertices are autonomous.  Let $\cP$ be a $\mathcal{V}$-poset on $n$ vertices.  If $\cP$ is the disjoint union of multiple posets then each of its connected subposets is a smaller $\mathcal{V}$-poset. By inductive hypothesis for each connected subposet there is a graph that generates it  and has only the pressing sequences dictated by said subposet.  It follows that in this case $\cP$ is autonomous as well.  Suppose now that $\cP$ is connected.  It then follows that $\cP$ has a unique maximal or a unique minimal element.  Let $\cP-x$ be the result of removing a unique maximal or minimal element from $\cP$.  Observe that $\cP-x$ is a $\mathcal{V}$-poset and thus by induction is autonomous; let $H$ be a graph such that $\mathfrak{S}(H)=\{\cP-x\}$.  By Lemmas \ref{left append observation} and \ref{right append observation}, $x\oplus H$ or $H\oplus x$  generates only $\cP$ and therefore  is  autonomous.
\end{proof}

\end{section}

\begin{section}{$\mathcal{V}$-poset Recognition}

For a poset $\cP$ we let $n_\cP$ and $e_\cP$ denote the number of vertices and  edges in the  Hasse diagram of the poset, respectively.  We let $h_\cP$ denote the sum of the heights of components of $\cP$ (the height of a poset is the length of its longest chain), $c_\cP$ denote the number of components of $\cP$, and $M_\cP$ and $m_\cP$ denote the number of maximal and minimal elements in $\cP$, respectively.
\begin{lemma}\label{edge count}
If $\cP$ is a $\mathcal{V}$-poset then
$$
e_\cP=2n_\cP+c_\cP-M_\cP-m_\cP-h_\cP \leq 2n_\cP - 2
$$
\end{lemma}

\begin{proof}  We show that $e_\cP=2n_\cP+1-M_\cP-m_\cP-h_\cP$ for a connected poset; the equality above follows by summing over components, and the inequality is immediate.
Observe that if $n_\cP=1$ then $\cP$ is a poset one element and hence $(2n_\cP+1)-(M_\cP+m_\cP+h_\cP)=0=e_\cP$.  Assume towards an inductive argument that $n_\cP\geq 2$.  Since $\cP$ is connected it must have a unique minimal or maximal element, say $x$, which we assume will be maximal (as the argument is identical for a minimal element). Let $\mathcal{Q}=\cP-\{x\}$.  Then, by applying the inductive hypothesis to $\mathcal{Q}$,  $$e_\cP-M_{\mathcal{Q}}=e_{\mathcal{Q}}=2n_\mathcal{Q}+1-M_\mathcal{Q}-m_\mathcal{Q}-h_\mathcal{Q}$$
$$e_\cP=2n_\mathcal{Q}+1-m_\mathcal{Q}-h_\mathcal{Q}=2(n_\cP-1)+1-m_\cP-(h_\cP-1)$$
$$e_\cP=2n_\cP-m_\cP-h_\cP$$
By noting that $M_\cP=1$, we have our result.
\end{proof}
We now give a different edge count that uses width (referred to as $w_\cP$ in the statement) instead of heights.  While both of these edge counts are necessary for the property of being a $\mathcal{V}$-poset, even when taken together, they are not sufficient.

\begin{lemma}\label{width count}
If $\cP$ is a $\mathcal{V}$-poset then
$$
e_\cP=n_\cP+w_\cP -M_\cP-m_\cP
$$
\end{lemma}

\begin{proof}  
As in the previous proof, we show that $e_\cP=n_\cP+w_\cP -M_\cP-m_\cP$ for a connected poset; the equality above follows by summing over components since the width of a disconnected poset is the sum of the width of its connected components (i.e. the length of a maximal antichain).
Observe that if $n_\cP=1$ then $\cP$ is a poset one element and $n_\cP+w_\cP -M_\cP-m_\cP=0=e_\cP$.  Assume towards an inductive argument that $n_\cP\geq 2$.  Since $\cP$ is connected it must have a unique minimal or maximal element, say $x$, which we assume will be maximal (as the argument is identical for a minimal element). Let $\mathcal{Q}=\cP-\{x\}$.  Then, by applying the inductive hypothesis to $\mathcal{Q}$,  
\begin{align*}
e_\cP&=e_\mathcal{Q}+M_\mathcal{Q}= (n_\mathcal{Q}+w_\mathcal{Q} -M_\mathcal{Q}-m_\mathcal{Q})+M_\mathcal{Q} \\
&= n_\mathcal{Q}+w_\mathcal{Q} -m_\mathcal{Q}= n_\mathcal{P}-1+w_\mathcal{P} -m_\mathcal{P}= n_\mathcal{P}-M_\mathcal{P}+w_\mathcal{P} -m_\mathcal{P}.
\end{align*}

\end{proof}

We propose an algorithm for the recognition of autonomous posets that operates on an arbitrary directed acyclic graph whose transitive closure is the poset in question. As a subroutine, we employ an algorithm found in \cite{valdes1979recognition} that detects if a directed acyclic graph contains an induced copy of the pattern $N$ and, if the input is found to be induced $N$-free, it also returns the transitive reduction of the input.  The aforementioned subroutine is guaranteed to run in $O(|V|+|E|)$. Observe that by the proof of Lemma \ref{bowties are not autonomous}, in order to determine if an induced $N$-free poset is a $\mathcal{V}$-poset we need only to verify that its transitive-reduction does not contain a sub-DAG that is isomorphic to $([4],\{(1,3), (2,3), (1,4), (2,4)\})$  (as done in Subroutine $2$) and does not contain sub-DAG whose transitive closure (interpreted as a poset) is isomorphic to $X(n)$, ($n\geq 5$).

Lemma \ref{edge count} shows that if we present the poset by the transitively-reduced directed acyclic graph with cover relations as edges then the run-time is  $O(|V|)$.
Observe that in Subroutines $2$ and $3$ each edge is traversed at most twice,  hence these algorithms have run-time $O(|V|+|E|)$.  Thus the presented algorithm has the same run-time as Subroutine $1$.
\vspace{3mm}

\noindent \textbf{Algorithm 1}
\begin{algorithmic}[1]
\STATE \textbf{input:} a directed acyclic graph $D$. 
\STATE \textbf{output:} $\TRUE$ or $\FALSE$. True if the transitive closure of $D$ is a $\mathcal{V}$-poset, False otherwise.
\STATE Bool $\gets \TRUE$
\IF{IsSeriesParallel(D)[Bool]=False}
\STATE Bool $\gets \FALSE$
\ELSE 
\STATE D $\gets$ IsSeriesParallel(D)[DAG]
\IF{IsBowtieFree(D)$=\FALSE$}
\STATE Bool $\gets\FALSE$
\ELSE
\IF{ClosureIsVPoset$(D)=\FALSE$}
\STATE  Bool $\gets \FALSE$
\ENDIF
\ENDIF
\ENDIF
\RETURN Bool
\end{algorithmic}\vspace{3mm}

\noindent \textbf{Subroutine 1: IsSeriesParallel()}
\begin{algorithmic}[1]
\STATE \textbf{input:} a directed acyclic graph $D$. 
\STATE \textbf{output:} $(\textrm{Bool},DAG)$.  Bool$=\TRUE$ when $D$ has a series-parallel decomposition and Bool$=\FALSE$ otherwise,  and $DAG$ is the transitive reduction of $D$.
\STATE Algorithm found in \cite{valdes1979recognition}
\RETURN $(\textrm{Bool},DAG)$
\end{algorithmic}
\vspace{3mm}

\noindent \textbf{Subroutine 2: IsBowtieFree()}
\begin{algorithmic}[1]
\STATE \textbf{input:} an induced $N$-free, transitively reduced directed acyclic graph $D$. 
\STATE \textbf{output:} $\TRUE$ or $\FALSE$. False if some induced subgraph of $D$ is isomorphic to the bowtie digraph $(\{a,b,c,d\}, \{(a,c), (a,d), (b,c), (b,d)\})$, True otherwise.
\STATE Bool $\gets\TRUE$, Current $\gets\emptyset$, Parents $\gets\emptyset$, Visited $\gets\emptyset$
\FOR{$v\in V(D)$}
\IF{OutDegree$(v)=0$}
\STATE Current.Add($v$)
\ENDIF
\ENDFOR
\WHILE{Current$\neq \emptyset$}
\FOR{$v \in$ Current}
\FOR{$u \in$ InNeighborhood($v$)}
\STATE Parents.Add($u$)
\ENDFOR
\ENDFOR
\FOR{$v \in$ Parents}
\IF{OutDegree$(v)>1$}
\FOR{$u \in$ OutNeighborhood($v$)}
\IF{InDegree$(u)>1$}
\STATE Bool $\gets \FALSE$ (Break {\bf while} loop)
\ENDIF
\STATE Visited.Add($u$)
\ENDFOR
\ENDIF
\ENDFOR
\FOR{$v\in$ Current}
\STATE Visited.Add($v$)
\ENDFOR
\STATE Current $\gets\emptyset$
\FOR{$v \in $Parents}
\IF{$v \notin$ Visited}
\STATE Current.Add($v$)
\STATE Visited.Add($v$)
\ENDIF
\ENDFOR
\STATE Parents $\gets\emptyset$
\ENDWHILE
\RETURN Bool
\end{algorithmic}
\vspace{3mm}

\noindent \textbf{Subroutine 3: ClosureIsVPoset()}
\begin{algorithmic}[1]
\STATE \textbf{input:} an induced $N$-free, induced bowtie-free, transitively reduced directed acyclic graph $D$. 
\STATE \textbf{output:} $\TRUE$ or $\FALSE$. True if the transitive closure of $D$ is a $\mathcal{V}$-poset, False otherwise.
\STATE Bool $\gets\TRUE$, Current $\gets\emptyset$, Parents $\gets\emptyset$, Visited $\gets\emptyset$, Multiple $\gets\emptyset$
\FOR{$v\in V(D)$}
\IF{OutDegree$(v)=0$}
\STATE Current.Add($v$)
\ENDIF
\ENDFOR
\WHILE{Current $\neq \emptyset$}
\FOR{$v \in$ Current}
\IF{$v\in$ Multiple $\AND$ InDegree$(v)>1$}
\STATE Bool $\gets\FALSE$ (Break {\bf while} loop)
\ENDIF
\FOR{$u \in$ InNeighborhood($v$)}
\STATE Parents.Add($u$)
\IF{$v$ in Multiple}
\STATE Multiple.Add($u$)
\ENDIF
\ENDFOR
\ENDFOR
\FOR{$v \in$ Parents}
\IF{OutDegree$(v)>1$}
\STATE Multiple.Add($v$)
\ENDIF
\ENDFOR
\FOR{$v \in$ Current}
\STATE Visited.Add($v$)
\ENDFOR
\STATE Current $\gets\emptyset$
\FOR{$v \in $Parents}
\IF{$v \notin$ Visited}
\STATE Current.Add($v$)
\ENDIF
\ENDFOR
\STATE Parents $\gets\emptyset$
\ENDWHILE
\RETURN Bool
\end{algorithmic}

\end{section}
\begin{section}{Open Questions}
In Lemma \ref{linear extensions are pressing sequences} we demonstrate that the successful pressing sequences of an OSP-graph are the linear extensions of a set of posets that arise from the instructional Cholesky roots of the graph.  An autonomous graph has the property that its successful pressing sequences are all linear extensions of a single poset.  In particular, in the autonomous case, this poset can be viewed as the intersection of all of the successful pressing sequences of the graph (interpreted as linear extensions).  In the case that the OSP-graph is not autonomous then the posets are the intersections of  pairwise disjoint families of successful pressing sequences.  
Thus, we have that if $G$ is an OSP-graph then the instructional posets of $G$ partition $\Sigma(G)$  into disjoint sets $S_1, S_2, \ldots, S_k$ satisfying that $\linext\left(\bigcap S_i \right)=S_i$ for each $i \in [k]$. Observe that this partition is not sufficient to determine the instructional posets of a graph since, for example $\linext\left(\bigcap \{\sigma\} \right)=\{\sigma\}$. 

\begin{question}
In general, how many distinct partitions of $\Sigma(G)$ into disjoint sets $\{S_i\}_i$ exist such that $\linext\left(\bigcap S_i \right)=S_i$ for each $i$?
\end{question}

The present work arose in the context of studying the complexity of enumeration of pressing sequences.  While every poset has a graph $G$ that generates it, only for the autonomous posets $\mathcal{P}$ does there exist a $G$ for which $\mathcal{P}$ is unaccompanied by other posets in $\mathfrak{S}(G)$. As we have shown that the autonomous posets are a subset of the series-parallel posets, this means that demonstrating \#P-hardness of counting pressing sequences or efficient sampling asymptotically uniformly at random from all pressing sequences of graph cannot be derived directly from results on the complexity of linear extension enumeration (see \cite{brightwell1991counting}).

\begin{conjecture} Exactly counting pressing sequences of a graph is \#P-hard.
\end{conjecture}
If exact counting is not possible then exhibiting an FPRAS would be desirable -- and is often possible for problems which are \#P-hard.  In the case that the number of posets generated by an OSP-graph is small (say, polynomial in the number of vertices), then it may be possible to adapt an  FPRAS for sampling linear extensions  (see \cite{jerrumsinclair1989}).

\begin{question}
Does there exists a constant $c$ such that $|\mathfrak{S}(G)|=O(n^c)$ for all graphs $G$ on $n$ vertices?
\end{question}
  
\begin{question}
Is there an FPRAS for counting the number of pressing sequences of a graph?
\end{question}

\end{section}

\newpage

\bibliographystyle{alpha}
\bibliography{bibliography}

\begin{thebibliography}{VTL79}

\bibitem[BW91]{brightwell1991counting}
Graham Brightwell and Peter Winkler.
\newblock Counting linear extensions.
\newblock {\em Order}, 8(3):225--242, 1991.

\bibitem[CD16]{cooper2016successful}
Joshua Cooper and Jeffrey Davis.
\newblock Successful pressing sequences for a bicolored graph and binary
  matrices.
\newblock {\em Linear Algebra and its Applications}, 490:162--173, 2016.

\bibitem[CW18]{cooper2017uniquely}
Joshua Cooper and Hays Whitlatch.
\newblock Uniquely pressable graphs: Characterization, enumeration, and
  recognition.
\newblock {\em Advances in Applied Mathematics, to appear. Preprint at
  arXiv:1706.07468}, 2018.

\bibitem[DS38]{dobzhansky1938inversions}
Th~Dobzhansky and Alfred~H Sturtevant.
\newblock Inversions in the chromosomes of drosophila pseudoobscura.
\newblock {\em Genetics}, 23(1):28, 1938.

\bibitem[HP99]{hannenhalli1999transforming}
Sridhar Hannenhalli and Pavel~A Pevzner.
\newblock Transforming cabbage into turnip: polynomial algorithm for sorting
  signed permutations by reversals.
\newblock {\em Journal of the ACM (JACM)}, 46(1):1--27, 1999.

\bibitem[HT17]{hasebe2017order}
Takahiro Hasebe and Shuhei Tsujie.
\newblock Order quasisymmetric functions distinguish rooted trees.
\newblock {\em Journal of Algebraic Combinatorics}, 46(3-4):499--515, 2017.

\bibitem[MW18]{misanantenaina2018tutte}
Valisoa~Razanajatovo Misanantenaina and Stephan Wagner.
\newblock A tutte-like polynomial for rooted trees and specific posets.
\newblock {\em arXiv preprint arXiv:1803.09623}, 2018.

\bibitem[SD36]{sturtevant1936inversions}
Alfred~H Sturtevant and Th~Dobzhansky.
\newblock Inversions in the third chromosome of wild races of drosophila
  pseudoobscura, and their use in the study of the history of the species.
\newblock {\em Proceedings of the National Academy of Sciences},
  22(7):448--450, 1936.

\bibitem[SJ89]{jerrumsinclair1989}
Alistair Sinclair and Mark Jerrum.
\newblock Approximate counting, uniform generation and rapidly mixing {M}arkov
  chains.
\newblock {\em Inform. and Comput.}, 82(1):93--133, 1989.

\bibitem[VTL79]{valdes1979recognition}
Jacobo Valdes, Robert~E Tarjan, and Eugene~L Lawler.
\newblock The recognition of series parallel digraphs.
\newblock In {\em Proceedings of the eleventh annual ACM symposium on Theory of
  computing}, pages 1--12. ACM, 1979.

\end{thebibliography}
\end{document}